\documentclass[preprint,1p,authoryear,times]{elsarticle}

\usepackage{amssymb}
\usepackage{amsmath}

\newcommand{\RR}{{\mathbb{R}}}
\newcommand{\NN}{{\mathbb{N}}}
\newcommand{\ZZ}{{\mathbb{Z}}}
\newcommand{\CC}{{\mathbb{C}}}
\newcommand{\TT}{{\mathbb{T}}}
\newcommand{\QQ}{{\mathbb{Q}}}
\newcommand{\cG}{{\mathcal{G}}}

\newcommand{\cP}{{\mathcal{P}}}
\newcommand{\cQ}{{\mathcal{Q}}}

\newcommand{\rank}{{\mathop{\rm rank}\nolimits\,}}
\newcommand{\diag}{{\mathop{\rm diag}\nolimits\,}}
\newcommand{\Span}{{\mathop{\rm span}\nolimits\,}}

\newtheorem{thm}{Theorem}
\newtheorem{lemma}[thm]{Lemma}
\newtheorem{corollary}[thm]{Corollary}
\newtheorem{problem}[thm]{Problem}
\newtheorem{proposition}[thm]{Proposition}
\newdefinition{remark}{Remark}
\newdefinition{definition}{Definition}
\newdefinition{example}{Example}
\newdefinition{algorithm}{Algorithm}
\newproof{proof}{Proof}

\begin{document}

\begin{frontmatter}
  
  \title{Prony's method in several
    variables: symbolic solutions by universal interpolation}

  \author{Tomas Sauer}
  \ead{Tomas.Sauer@uni-passau.de}
  \address{Lehrstuhl f\"ur Mathematik mit Schwerpunkt Digitale
    Bildverarbeitung \& FORWISS,
    Universit\"at Passau, Innstr. 43, D-94053
    Passau, Germany.}

\begin{abstract}
  The paper considers a symbolic approach to Prony's method in several
  variables and its close connection to multivariate polynomial
  interpolation. Based on the concept of universal interpolation that
  can be seen as a weak generalization of univariate Chebychev
  systems, we can give estimates on the minimal number of evaluations
  needed to solve Prony's problem.
\end{abstract}


\end{frontmatter}

\section{Introduction}
Formulated in several variables, \emph{Prony's problem} consists of
reconstructing a function $f : \RR^s \to \CC$ of the form
\begin{equation}
  \label{eq:PronyFun}
  f(x) = \sum_{\omega \in \Omega} f_\omega e^{\omega^T x}, \qquad
  f_\omega \in \CC \setminus \{ 0 \}, \quad \omega \in \Omega \subset
  \left( \RR + i \TT \right)^s,
\end{equation}
from discrete samples,
where, as usual, $\TT$ stands for the \emph{torus} $\TT := \RR / 2 \pi
\ZZ$. The restriction of the imaginary part of the frequencies is
needed to avoid ambiguities in the solution.
In one variable, this problem and its solution date back to Prony
\citep{prony95:_essai} in 1795 and since then various numerical
methods have been 
devised to solve the problem, in particular the ESPRIT and MUSIC 
algorithms \citep{roy89:_esprit,schmidt86:_multip} from multi source
radar detection with extensions to higher dimensions on grids in
\citep{rouquette01:_estim,yilmazer06:_matrix}. 
One should also consider
\citep{potts15:_fast_esprit} 
for recent improvements and \citep{plonka14:_prony} for a survey on
Prony's method and its extensions and generalizations. The
\emph{matrix pencil} approach for the Hankel matrices has
also been considered in
\citep{hua90:_matrix_pencil_method_estim_param,hua92:_estim} in one and
two variables.

The closely related problem of reconstruction of sparse polynomials in
several variables has also been considered by methods other than
Prony's approach. Probabilistic methods with a small expected number
of evaluations can be found in \citep{zippel79:_probab}, see also
\citep{zippel90:_inter} and, more recently, in \citep{kaltofen16:_spars}.

The multivariate version of Prony's method has gained some popularity
recently and was approached by projection methods as in
\citep{diederichs15:_param_estim_bivar_expon_sums,potts13:_param}, as
well as by more or less direct multivariate attempts
\citep{kunis16:_prony}. This
paper is an extension of \citep{Sauer15S}, where the algebraic nature of
the multivariate problem has been pointed out, resulting in a fast
numerical method based on orthogonal H--bases.

Let us begin with a slightly informal presentation of the algebraic
structure underlying Prony's problem in several variables:
the approach consist of considering
finite parts of the infinite \emph{Hankel matrix}
\begin{equation}
  \label{eq:InfiniteHankel}
  F := \left[ f(\alpha+\beta) : \alpha,\beta \in \NN_0^s \right],  
\end{equation}
for example
\begin{equation}
  \label{eq:FnDef}
  F_n := \left[ f(\alpha+\beta) : \alpha,\beta \in \Gamma_n \right],
  \qquad \Gamma_n := \left\{ \alpha \in \NN_0^s : |\alpha| \le n \right\},
\end{equation}
with the standard \emph{length} $|\alpha| = \alpha_1 + \cdots +
\alpha_s$ of a \emph{multiindex} $\alpha \in \NN_0^s$. The crucial
observation is that the \emph{Prony ideal} $I_\Omega$, the set of all
polynomials vanishing on
\[
X_\Omega = e^\Omega = \left\{ x_\omega = e^\omega = \left(
    e^{\omega_1},\dots,e^{\omega_s} \right) : \omega \in \Omega
\right\},
\]
is in one-to-one correspondence with the kernels of the matrices $F_n$. More
precisely, if we denote by $\Pi_n$ the vector space of all polynomials
of \emph{total degree} at most $n$, and identify an polynomial
\[
\Pi_n \ni p(x) = \sum_{|\alpha| \le n} p_\alpha \, x^\alpha
\]
with its coefficient vector $p = \left[ p_\alpha : \alpha \in \Gamma_n
\right] \in \CC^{\Gamma_n}$, then the following result from
\citep{Sauer15S} holds true.

\begin{thm}\label{T:FnKernel}
  For sufficiently large $n$ we have that $p \in I_\Omega \cap \Pi_n$
  if and only if $F_n p = 0$.
\end{thm}

\noindent
The ``sufficiently large'' in Theorem~\ref{T:FnKernel} can be made
concrete: it is the degree of some degree reducing interpolation space
for $X_\Omega$ and since all degree reducing interpolation spaces have
the same degree, this is indeed a \emph{geometric} quantity depending
on $X_\Omega$ only.
The algorithm developed in \citep{Sauer15S} relies on the a
priori knowledge of such a sufficiently large $n$ and then
successively builds the matrices
\[
F_{n,k} := \left[ f(\alpha+\beta) :
  \begin{array}{c}
    \alpha \in \Gamma_n \\ \beta \in \Gamma_k 
  \end{array}
\right], \qquad k=0,\dots,n,
\]
for which the following observation has been made in \citep{Sauer15S}.

\begin{thm}\label{T:HilbertFun}
  For sufficiently large $n$, we have that
  \begin{enumerate}
  \item $\ker F_{n,k} \simeq I_\Omega \cap \Pi_k$, $k = 0,\dots,n$,
  \item $k \mapsto \rank F_{n,k}$ computes the \emph{affine Hilbert
      function} for the ideal $I_\Omega$,
  \item there exists a number $0 \le m \le n$ such that
    \[
    \rank F_{n,0} < \cdots < \rank F_{n,m-1} = \rank F_{n,m} = \cdots
    = \rank F_{n,n} = \rank F_n,
    \]
    and this $m$ is the minimal choice for $n$.
  \end{enumerate}
\end{thm}

\noindent
Based on these structural observations, an algorithm can be derived
that computes an orthogonal H--basis in the sense of \citep{Sauer01} as
well as a graded homogeneous basis for the interpolation space $\Pi /
I_\Omega$ entirely by application of standard techniques from
Numerical Linear Algebra, namely singular valued decompositions and
$QR$ factorizations. This allows for fast and accurate solutions of even
high dimensional problems in floating point arithmetic.

Once the H--basis and the interpolation
space are determined, it is a fairly standard approach, see
\citep{GonzalesRouillierRoyTrujillo99,Stetter95} to determine the 
\emph{multiplication tables}, a set of $s$ commuting matrices of size
$\# \Omega \times \# \Omega$ whose eigenvalues are the components of
the $x_\omega$ that can be related by the respective eigenvectors, see
\citep{moeller01:_multiv}.

This paper takes a somewhat different approach to Prony's problem by
considering interpolation spaces
spanned by a minimal number of monomials, see
\citep{deBoor07,Sauer97a}. While orthogonal H--bases are more favorable
from a numerical point of view and work well in a numerical
environment, the underlying methods from Numerical Linear Algebra, in
particular orthogonal factorizations like the $QR$ decomposition of
matrices, cause difficulties in a symbolic framework due to the
occurrence of square roots. In contrast to those numerical methods,
this paper studies a \emph{symbolic} approach
that will
provide us with a \emph{minimal} sampling set for
Prony's method and tell us what an (asymptotically) minimal number of
evaluations of $f$ needed for the reconstruction. In doing so, we will
also gain some further inside into the algebraic structure of Prony's
problem.

The paper is organized as follows. In Section~\ref{sec:PronyRevis} the
notation will be fixed and Prony's problem will be expressed in terms
of degree reducing interpolation. In Section~\ref{sec:UniInter} we
study the fundamental algebraic tool, namely \emph{uniform
  interpolation}. This means the identification of spaces that permit
interpolation at any subset of $\CC^s$ of a given cardinality. Based
on this concept, Section~\ref{sec:MinRecovery} points out how to
solve Prony's problem with a minimal number of evaluations. Detailed
symbolic algorithms for that purpose are developed in
Section~\ref{sec:SymbolAlgo}. Finally, Section~\ref{sec:SparsePoly}
briefly points out the connection to sparse polynomials and how those
can be determined symbolically and an appendix in
Section~\ref{sec:Appendix} provides two valuable tools from
computational ideal theory together with proofs.

All results presented in this paper are of algebraic nature. Of
course, the numerical stability of the methods and of the
reconstruction in general depends on the conditioning of Vandermonde
matrices. This important and valuable question, however, is not in the
scope of this paper.

\section{Prony's problem revisited}
\label{sec:PronyRevis}
Let $\Pi = \CC [x_1,\dots,x_s]$ denote the algebra of polynomials in
$s$ variables with complex coefficients. We consider the
\emph{nonnegative integer grid}
$\Gamma = \NN_0^s$. For a finite $A \subset
\Gamma$ we define
\[
\Pi_A := \left\{ p(x) = \sum_{\alpha \in A} p_\alpha \, x^\alpha :
  p_\alpha \in \CC \right\}
\]
as the space of polynomials \emph{supported} on $A$, a finite
dimensional subspace of $\Pi$ of dimension $\# A$. Recall that
the \emph{total degree} of a polynomial $p \in \Pi$
is defined as
\[
\deg p := \max \{ |\alpha| : p_\alpha \neq 0 \}.
\]
In the important
case $A = \Gamma_n = \{ \alpha : |\alpha| \le n \}$, we use the common
abbreviation $\Pi_n := \Pi_{\Gamma_n}$ for the vector space of all
polynomials of total degree (at most) $n$. In an analogous way, we define
$\Gamma_n^0 = \{ \alpha : |\alpha | = n \}$ as the set of
\emph{homogeneous} multiindices of length $n$ and $\Pi_n^0 :=
\Pi_{\Gamma_n^0}$ as the \emph{homogeneous} polynomials of degree
(exactly) $n$.

The \emph{coefficients}
$p_\alpha$ of $p \in \Pi_A$ can be conveniently arranged into a vector
$p = ( p_\alpha : \alpha \in A ) \in \CC^A$; we use the same symbol
for the polynomial and the vector, the respective meaning will be clear
from the context. 
The next notion is standard in (polynomial)
interpolation theory.

\begin{definition}\label{D:Vandermonde}
  For finite sets $A \subset \Gamma$ and $X \subset \CC^s$ we define
  the \emph{Vandermonde matrix} $V(X,A)$ as
  \[
  V(X,A) := \left[ x^\alpha :
    \begin{array}{c}
      x \in X \\ \alpha \in A
    \end{array}
  \right].
  \]
\end{definition}

\noindent
Vandermonde matrices play a fundamental role in Prony's method as the
following simple computation shows: for $A,B \subset \Gamma$ we
define the Hankel matrix
\begin{equation}
  \label{eq:FABDef}
  F_{A,B} := \left[ f (\alpha+\beta) :
    \begin{array}{c}
      \alpha \in A \\ \beta \in B
    \end{array}
  \right].
\end{equation}
Using the unit vectors $e_\alpha := ( \delta_{\alpha,\alpha'} : \alpha'
\in A ) \in \CC^A$, we find that
\begin{eqnarray*}
  ( F_{A,B} )_{\alpha,\beta}  = e_\alpha^T F_{A,B} e_\beta
  & = & \sum_{\omega \in \Omega} f_\omega \,
  e^{\omega^T ( \alpha + \beta )} = \sum_{\omega \in \Omega} f_\omega \,
  e^{\omega^T \alpha} e^{\omega^T \beta} \\
  & = & \sum_{\omega \in \Omega} e_\alpha ^T V ( X_\Omega,A )^T
  e_\omega \, f_\omega \, e_\omega^T V ( X_\Omega,B ) e_\beta,
\end{eqnarray*}
which yields the well--known factorization
\begin{equation}
  \label{eq:FABFactor}
  F_{A,B} = V(X_\Omega,A)^T \, F_\Omega \, V (X_\Omega,B), \qquad
  F_\Omega := \diag ( f_\omega : \omega \in \Omega ),
\end{equation}
already used in the univariate ESPRIT method
\citep{roy89:_esprit}. Since, by assumption \eqref{eq:PronyFun},
$f_\omega \neq 0$, $\omega \in \Omega$, the rank of $F_{A,B}$ is
at most $\# \Omega$ with equality if and only if
\begin{equation}
  \label{eq:rankFCond}
  \rank V(X_\Omega,A) = \rank V(X_\Omega,B) = \# \Omega.  
\end{equation}
The meaning of \eqref{eq:rankFCond} is well--known: $\Pi_A$ and
$\Pi_B$ have to be interpolation spaces for $X_\Omega$.

\begin{definition}\label{D:InterpolSpace}
  A subspace $\cP$ of $\Pi$ is called an \emph{interpolation space}
  for $X$ if for any $y \in \CC^X$ there exists (at
  least one) $p \in \cP$ such that $p(X) = y$.
\end{definition}

\noindent
Despite of its simple derivation, \eqref{eq:rankFCond} has an
immediate important consequence for Prony's problem and the
reconstruction of $\Omega$ from $F_{A,B}$

\begin{thm}\label{T:ReconTheo}
  The coefficients $f_\omega$ can be reconstructed from $F_{A,B}$ if
  and only if $\Pi_A$ and $\Pi_B$ are interpolation spaces with
  respect to $X_\Omega$.
\end{thm}

\begin{proof}
  Suppose that $\Pi_A$ and $\Pi_B$ are interpolation spaces with
  respect to $X_\Omega$, then $\# A \ge \# \Omega$, and there exist
  coefficient vectors $p_\omega = ( p_{\omega,\alpha} : \alpha \in A
  )$ such that for $\omega' \in \Omega$
  \[
  \delta_{\omega,\omega'} = p_\omega ( x_{\omega'} ) = \sum_{\alpha
    \in A} p_{\omega,\alpha} \, x_{\omega'}^\alpha,
  \]
  hence
  \[
  V (X_\Omega,A) \, \left[ p_{\omega,\alpha} :
    \begin{array}{c}
      \alpha \in A \\ \omega \in \Omega
    \end{array}
  \right] = I_{\# \Omega}.
  \]
  In other words, this matrix is a right inverse of $V (X_\Omega,A)$
  which we call $V (X_\Omega,A)^{-1}$ and since the same holds for $V
  (X_\Omega,B)$, it follows that
  \[
  V (X_\Omega,A)^{-T} \, F_{A,B}  V (X_\Omega,B)^{-1} = F_\Omega,
  \]
  which reconstructs the coefficients under the assumption that
  $\Pi_A$ and $\Pi_B$ are interpolation spaces for $X_\Omega$.

  Conversely, if $\rank V
  (X_\Omega,B) < \# \Omega$, then there exists a nonzero diagonal
  matrix $F \in \CC^{\Omega \times \Omega}$ such that $F \, V
  (X_\Omega,B) = 0$ and therefore
  \[
  F_{A,B} = V(X_\Omega,A)^T \, ( F_\Omega + F ) \, V (X_\Omega,B)
  \]
  so that $F_\Omega$ cannot be reconstructed from $F_{A,B}$. An analogous
  argument can also be used in the case that $\rank V (X_\Omega,A) <
  \# \Omega$.
\end{proof}

\begin{corollary}\label{C:PronyAB}
  Any sampling sets $A,B$ for Prony's method must be chosen such that
  $\Pi_A$ and $\Pi_B$ are interpolation spaces for $X_\Omega$.
\end{corollary}

\begin{definition}
  A subspace $\cP$ of $\Pi$ is called a \emph{degree reducing
    interpolation space} for a finite set $X \subset \CC^s$ if for any
  $q \in \Pi$ there exists a \emph{unique} polynomial $p \in \cP$ such
  that $p(X) = q(X)$ and $\deg p(X) \le \deg q(X)$.
\end{definition}

\noindent
Degree reducing interpolation spaces for some set $X \subset \CC^s$ of
\emph{sites} have the advantage that they give
the ideal $I_X = \{ p \in \Pi : p(X) = 0 \}$ almost for free.
To be more concrete, let $A \subset \Gamma$ be
such that $\Pi_A$ is a degree reducing interpolation space for
$X$ and let $L_A : \Pi \to \Pi_A$ denote the \emph{interpolation
  operator} defined by
\begin{equation}
  \label{eq:InterpolOp}
  ( L_A p ) (X) = p(X)
\end{equation}
which is well defined because degree reducing interpolation is unique
by definition. It can then be shown that the polynomials
\begin{equation}
  \label{eq:IdealBasisDef}
  h_\alpha := (\cdot)^\alpha - L_A (\cdot)^\alpha, \qquad \alpha \in
  A^c := \Gamma \setminus A,
\end{equation}
form an \emph{H--basis} of $I_X = \{ q : q(X) = 0 \}$, that is, any
polynomial $q \in I_X$ can be written as
\begin{equation}
  \label{eq:HbasisRep}
  q = \sum_{\alpha \in A^c} q_\alpha \, h_\alpha, \qquad \deg q_\alpha
  \le \deg q - \deg h_\alpha,
\end{equation}
where the sum in \eqref{eq:HbasisRep} is finite and uses the
convention that $\deg p < 0$ iff $p = 0$. This fact, that is some
folklore in ideal interpolation will be (re-)proved in even stronger form
in Lemma~\ref{L:InterpolDuality} of the abstract.

By Corollary~\ref{C:PronyAB} we can reconstruct $F_\Omega$ from the
symmetric sampling matrix
$F_{A,A}$ if and only if $\Pi_A$ is an interpolation space. $A$
can be chosen minimally or at least of minimal cardinality by
requesting $\Pi_A$ to be a minimal degree interpolation space. In
other words, computing a solution to Prony's problem with a minimal
number of evaluations leads to determining such a space
for the \emph{unknown} point set $X_\Omega$ from $F_{A,A}$. The
following two examples illustrate what can happen.

\begin{example}[Generic case]\label{Ex:Generic}
  Suppose for simplicity that
  \[
  N := \# \Omega = r_n = \dim \Pi_n = {n+s \choose s}.
  \]
  In this situation the set of all point configurations $X$ such that
  $\Pi_n$ is \emph{the} degree reducing interpolation space is open
  and dense in $( \CC^s )^{\# \Omega}$, hence,
  \[
  \det V ( X_\Omega, \Gamma_n ) \neq 0
  \]
  with probability $1$. Hence, $F_\Omega$ can be reconstructed from
  $F_n$ and the kernel of
  \[
  F_{n,n+1} = \left[ f(\alpha+\beta) :
    \begin{array}{c}
      \alpha \in \Gamma_n \\ \beta \in \Gamma_{n+1}
    \end{array}
  \right]
  \]
  determines an H--basis of $I_\Omega$
  by Theorem~\ref{T:FnKernel}. Hence, Prony's problem can be solved
  based on the knowledge of $f$ on the grid
  \[
  \left\{ \alpha + \beta : \alpha \in \Gamma_n, \, \beta \in
    \Gamma_{n+1} \right\} = \Gamma_{2n+1} 
  \]
  since any multiindex of length $2n+1$ can be written as the sum of
  two multiindices, one of length $n$, one of length $n+1$. Hence the
  number of samples is 
  $r_{2n+1}$ and since
  \begin{eqnarray*}
    \frac{r_{2n+1}}{N} & = & \frac{{2n+1+s \choose s}}{{n+s \choose
        s}}
    = \frac{(2n+1+s) \cdots (2n+2)}{(n+s) \cdots (n+1)}
    = \prod_{j=1}^s \left( 1 + \frac{n+1}{n+j} \right) \le 2^s,
  \end{eqnarray*}
  with $2^s$ being the smallest bound independent of $n$,
  it follows that the generic case needs $\approx 2^s \# \Omega$
  samples of $f$ without any \emph{curse of dimensionality} in $\# \Omega$.
\end{example}

\noindent
Unfortunately, not every configuration of the frequencies $\Omega$ and
therefore of the points $X_\Omega$ is generic and even if any
configuration could be made
generic by an arbitrarily small perturbation, relying on the generic
situation leads to numerical and structural problems. The second
example shows that linear complexity cannot be expected in general.

\begin{example}[Hyperbola]\label{Ex:Hyperbola}
  Let $\Omega \subset \CC^2$ consist of $2N+1$ distinct frequencies of
  the form $(\omega_j,-\omega_j)$ with $\omega_j \in \RR$,
  $j=0,\dots,2N+1$. Then the points $x_j = ( e^{\omega_j}, e^{-\omega_j}
  ) \in \RR^2$ all lie on the hyperbola $x_1 x_2 = 1$, hence $q(x) =
  x_1 x_2 - 1 \in I_\Omega$, and the (unique) degree reducing
  interpolation 
  space is spanned by $1,x_1,\dots,x_1^N,x_2,\dots,x_2^N$
  which is a subset of $\Pi_N$. Hence, $A = \{ 0,\epsilon_1,\dots,
  N\epsilon_1, \epsilon_2, \dots, N\epsilon_2 \}$ where $\epsilon_j$
  stands for the unit multiindex. With $A' := A \cup \{ (N+1)
  \epsilon_1, (N+1) \epsilon_2 \}$ the minimal sampling set consists
  of
  \[
  A+A' = \{ \alpha + \beta : \alpha \in A,\beta \in A' \} \supset \{ \alpha
  \in \Gamma : \| \alpha \|_\infty \le N \},
  \]
  which has $> (N+1)^2$ elements and therefore at least $O(N^2)$ sampling
  points have to be used in this case.
\end{example}

\noindent
Nevertheless, since $\# (A+A) \le (\# A)^2$ we could always have a
chance to reconstruct $f$ from $O(N^2)$ samples, independent of the
dimension, provided we could find a set $A \subset \Gamma$ such that
$\Pi_A$ is a degree reducing interpolant. Why ``degree reducing'' is
so important will become clear in Section~\ref{sec:SymbolAlgo} where
we use this property to identify the ideal.

\section{Universal interpolation spaces}
\label{sec:UniInter}
The observations from the preceding section naturally suggest the
following question.
\begin{problem}\label{Pm:MinMon}
  For any $N \ge 0$ determine a \emph{minimal} set $\Upsilon_N \subset
  \Gamma$ such that for any set $X$ with $\#X = N$ we find a subset $A
  \subset \Upsilon_N$ such that $\Pi_A$ is a degree reducing
  interpolation space for $X$.
\end{problem}

\noindent
In one variable, we know that $\Upsilon_N = \{ 0,\dots,N-1\}$ solves
Problem~\ref{Pm:MinMon} since the polynomials $\Pi_{N-1}$ of degree at
most $N-1$ 
form a \emph{Chebychev system} of order $N$ or span a \emph{Haar
  space} of dimension $N$. Both means the same: any interpolation
problem at $N$ sites has a unique solution. Since Haar spaces
only exist in one variable according to Mairhuber's theorem, cf.
\citep{Lorentz66}, a solution of Problem~\ref{Pm:MinMon} must contain
more than $N$ elements.

\begin{remark}
  Problem~\ref{Pm:MinMon} is a simpler version of the classical problem of
  finding for any $N$ a polynomial space of \emph{minimal dimension} that
  allows for interpolation at arbitrary $N$ points. To my knowledge,
  these spaces are only known for some very special cases.
\end{remark}

\noindent
In the following definition of universal interpolation spaces we
define the data to be interpolated by means of polynomials.
Since $\Pi_N$ is always a universal interpolation
space of order $N+1$, cf. \citep{Sauer98}, this is no restriction, but
more consistent when degree reducing interpolation is concerned.

\begin{definition}\label{D:UniversalInterpolation}
  A subspace $\cP$ is called a \emph{universal interpolation space} or
  \emph{generalized Haar space} of order $N$ if for any $X \subset
  \CC^s$ with $\# X \le N$ and any $q \in \Pi$ there exists $p \in
  \cP$ such that $p(X)
  = q(X)$. $\cP$ is called a \emph{degree reducing universal
    interpolation space} if the interpolant $p \in \cP$ can be chosen such
  that $\deg p \le \deg q$.
\end{definition}

\begin{definition}\label{D:redundant}
  A degree reducing universal interpolation space $\cP$ of order $N$ is called
  \emph{redundant} if there exists a proper subspace $\cQ \subset \cP$
  that is also a degree reducing universal interpolation space.
\end{definition}

\begin{lemma}\label{L:DegRedContained}
  Any non-redundant degree reducing universal interpolation space
  $\cP$ of order $N+1$ is a subspace of $\Pi_N$.
\end{lemma}

\begin{proof}
  If $\cP$ is not a subspace of $\Pi_N$ and not redundant, there must be
  a configuration $X \subset \CC^s$ of $N$ sites such that at least
  one of the polynomials $\ell_x$ defined by $\ell_x (x') =
  \delta_{x,x'}$, $x,x' \in X$, has degree $> N$ as otherwise $\cP$
  could be chosen as a subspace of $\Pi_N$. Let $\ell_x$ denote this
  polynomial and $x$ the respective element of $X$. On the other hand,
  since $\Pi_N$ is also a universal interpolation space there exists
  $\tilde \ell_x \in \Pi_N$ with the same interpolation property
  $\tilde \ell_x (x') = \delta_{x,x'}$, $x' \in X$, so that $\ell_x$
  is the interpolant to $\tilde \ell_x$. But then $\deg \ell_x > N \ge
  \deg \tilde \ell_x$ contradicts the assumption that $\cP$ is degree reducing.
\end{proof}

\noindent
Returning to Problem~\ref{Pm:MinMon}, 
we now give an explicit non--redundant and therefore minimal monomial
degree reducing interpolation space that is 
spanned by monomials, hence is of the form $\Pi_A$ for some set $A
\subset \Gamma$. This can be formalized as follows.

\begin{definition}
  A set $A \subset \Gamma$ is called a \emph{monomial degree reducing
    universal interpolation set} of order $N$ if for any $X \subset
  \CC^s$ with $\# X \le N$ the polynomial space $\Pi_A$ is a degree
  reducing universal interpolation space.
\end{definition}

\noindent
To identify monomial degree reducing universal interpolation spaces,
we need another fundamental concept.

\begin{definition}\label{D:LowerSet}
  For two multiindices $\alpha,\beta \in \Gamma$ we write $\alpha \le
  \beta$ if $\alpha_j \le \beta_j$, $j=1,\dots,s$. We call $A \subset
  \Gamma$ a \emph{lower set} if $\alpha \in A$ implies
  $\{ \beta \in \Gamma : \beta \le \alpha \} \subset A$.
  By $L(\Gamma)$ we denote the set of all lower
  sets in $\Gamma$ and
  \[
  L_j (A) = \{ B \in L(\Gamma) : \# B = j \}, \qquad j \in \NN,
  \]
  stands for all lower sets of cardinality $j$.
\end{definition}

\begin{definition}\label{D:BorderCorona}
  For a finite set $A \subset \Gamma$ we define its \emph{border} as
  \begin{equation}
    \label{eq:DBorder}
    \partial A := \left( \bigcup_{j=1}^s \left( A + \epsilon_j \right) \right)
    \setminus A
  \end{equation}
  and its \emph{corona} as
  \begin{equation}
    \label{eq:DCrona}
    \lceil A \rceil := A \cup \partial A.
  \end{equation}
\end{definition}


\noindent
We can now describe degree reducing sets of monomials in terms of
lower sets.

\begin{thm}\label{T:MinInterpLS}
  If $A \subset \Gamma$ is a monomial degree reducing
  universal interpolation set of order $N$, then
  \begin{equation}
    \label{eq:MinInterpLS}
    A \supseteq \bigcup_{j=1}^N \bigcup_{B \in L_j (\Gamma)} B.  
  \end{equation}
\end{thm}

\begin{proof}
  Let $B$ be any lower set with $\# B \le N$ and consider
  interpolation on the \emph{grid}
  \[
  X_B = \left\{ \beta : \beta \in B \right\}.
  \]
  Since $A$ is a monomial degree reducing set, there exists a
  $B' \subseteq A$ such that $\Pi_{B'}$ is a degree reducing
  interpolation space for $X_B$ where uniqueness of interpolation
  implies that $\# B' = \# B$. Since $\Pi_{B'}$ is degree reducing,
  the polynomials
  \[
  (\cdot)^\beta - L_{B'} (\cdot)^\beta, \qquad \beta \in \partial B',
  \]
  form an H--Basis of $I_{X_B'}$, see Lemma~\ref{L:InterpolDuality},
  and $\Pi_{B'}$ is the \emph{normal form} or \emph{reduced} space for
  $X_B$ with respect to this H--basis and an appropriate inner
  product, see Lemma~\ref{L:MonHBasReduced}. The same holds
  true for $B$ and the standard inner product, 
  hence $\Pi_B$ is also a normal form space. It has
  been proved in \citep{Sauer04} that for interpolation grids based on
  lower sets the normal form interpolation space is unique
  independently of grading and inner product, hence
  $\Pi_B = \Pi_{B'}$ and since both spaces are spanned by $\# B$
  monomials, it follows that $B = B'$.
\end{proof}

\begin{proposition}\label{P:MinInterpLS}
  Any set $A$ that satisfies \eqref{eq:MinInterpLS} is a monomial
  degree reducing universal interpolation set of order $N$.
\end{proposition}

\begin{proof}
  Given $X \subset \CC^s$, $\# X \le N$,
  let $\cG$ be any reduced graded Gr\"obner basis for $I_X$ and $B$
  the index set for the quotient space $\Pi_B \simeq \Pi / I_X$. Since
  the complement of $B$ consists of the \emph{upper set} leading terms
  of the ideal, 
  cf. \citep[p.~230]{CoxLittleOShea92}, it follows that $B$ is a
  lower set and unique interpolation requests that $\# B = \#
  X$. Therefore, any interpolation problem with $\le N$ sites can be
  solved by an appropriate lower set $B$ of the same cardinality.
\end{proof}

\noindent
Combining Theorem~\ref{T:MinInterpLS} with
Proposition~\ref{P:MinInterpLS}, we immediately get the following
result.

\begin{corollary}\label{C:MinimalIntSet}
  The index set
  \begin{equation}
    \label{eq:MinimalIndexSet}
    A_N^* := \bigcup_{j=1}^N \bigcup_{B \in L_j (\Gamma)} B
  \end{equation}
  consisting of the union of all lower sets of cardinality at most $N$
  is a non--redundant, hence minimal, monomial degree reducing
  universal interpolation set.
\end{corollary}

\noindent
The index set $A_*$ defined in \eqref{eq:MinimalIndexSet} can easily be
described in a different way.

\begin{lemma}\label{L:A*Formula}
  For any $N \ge 1$ we have
  \begin{equation}
    \label{eq:A*Formula}
    \alpha \in A_N^* \qquad \Leftrightarrow \qquad
    \pi (\alpha) := \prod_{j=1}^s (\alpha_j + 1) \le N.
  \end{equation}
\end{lemma}

\begin{proof}
  Since the set $\{ \beta \in \Gamma : \beta \le \alpha \}$ has
  cardinality $(\alpha_1 + 1) \cdots (\alpha_s + 1)$, any multiindex
  satisfying the right hand side of \eqref{eq:A*Formula} determines a
  lower subset of $A_N^*$ of cardinality $\le N$, and therefore must
  belong to $A_N^*$. If, on
  the other hand $\pi (\alpha) > N$ then any lower set containing
  $\alpha$ contains $> N$ elements and cannot be a subset of $A_N^*$,
  hence $\alpha \not\in A_N^*$.
\end{proof}

\begin{definition}
  The set $\Upsilon_N := \{ \alpha \in \Gamma : \pi (\alpha) \le N \}$ is
  called the (positive octant of the) \emph{hyperbolic cross} of order $N$.
\end{definition}

\noindent
In the sequel we will use ``hyperbolic cross'' for the positive octant
since we only consider subsets of $\Gamma$. Hyperbolic crosses
play an important role in the context of
FFT methods \citep{doehler09:_noneq_fourier}, but also in general
multivariate Approximation Theory, cf. the recent survey
\citep{DungetAlHyperbolicCross}. In terms of interpolation, we can
summarize the above results as follows.

\begin{corollary}\label{C:HyperCrossMindeg}
  $\Upsilon_N$ forms a minimal monomial degree reducing universal
  interpolation set of order $N$.
\end{corollary}

\noindent
Since $\# \Upsilon_N \le N \log^{s-1} N$,
cf. \citep[Lemma~1.4, p.~71]{lubich08:_from_quant_class_molec_dynam},
the universal interpolation space based on $\Pi_{\Upsilon_N}$ allows for
degree reducing interpolation at arbitrary $N$ sites in $\CC^s$ which
can be seen as some ``Haar space without uniqueness'' where the number
of elements only exceeds the number of points by a very moderate
logarithmic factor.

\begin{remark}
  The hyperbolic crosses $\Upsilon_N$ are only minimal \emph{monomial}
  interpolation spaces, but not general ones. This can be easily seen
  in the case $N=3$ in $s=2$ where $\{ 1,x,y,x^2+y^2 \}$ forms a
  universal interpolation of minimal dimension $4$ while $\Upsilon_2 = \{
  (0,0), (1,0), (2,0), (0,1), (0,2) \}$ already consists of $5$
  elements and therefore has dimension $5$.
\end{remark}

\section{Minimal recovery}
\label{sec:MinRecovery}
By the results of the preceding section, a simple method can be
devised to solve Prony's problem, provided that $N = \# \Omega$ is
given:
\begin{enumerate}
\item Set up the matrix
  \[
  F = \left[ f(\alpha+\beta) :
    \begin{array}{c}
      \alpha \in \Upsilon_N \\ \beta \in \lceil \Upsilon_N \rceil
    \end{array}
  \right].
  \]
\item Compute the kernel of this matrix and therefore the \emph{Prony
    ideal} as the ideal generated by the kernel vectors, interpreted
  as polynomial coefficients.
\item Compute a graded Gr\"obner basis or an H--basis as in
  \citep{MoellerSauerSM00I}  and the normal form space,
  cf. \citep{Sauer06a}, for this ideal.
\item Determine the multiplication tables and their eigenvalues and
  therefore $X_\Omega$ as described in \citep{Sauer15S}.
\item Determine the coefficients by solving the Vandermonde system.
\end{enumerate}
This procedure is an evaluation efficient way of solving Prony's
problem where the number of function evaluations needed depends only
in a very mild way on the dimension $s$.

\begin{thm}\label{T:PronyMinimal}
  Prony's problem can be solved for $N$ frequencies in $\CC^s$ on the
  basis of at most 
  \[
  (s+1) N^2 \log^{2s-2} N
  \]
  point evaluations on the grid $\Gamma$.
\end{thm}

\begin{proof}
  The theorem relies on two simple facts: any kernel element of $F$
  belongs the ideal $I_\Omega$ by Theorem~\ref{T:PronyIdeal}, and by
  Lemma~\ref{L:InterpolDuality} from the appendix, the corona of
  $\Upsilon_N$ contains even
  an H--basis of the ideal $I_\Omega$, hence the 
  ideal and the quotient space can be determined from
  $F$. This allows for the reconstruction of
  frequencies by eigenvalues methods and coefficients by solving a
  simple linear system. Since the corona of a set contains at most
  $s+1$ times as many elements as the set, the estimate $\# \Upsilon_N \le
  N \log^{s-1} N$ leads to the claim.
\end{proof}

\begin{remark}
  The order $N^2$ is optimal up to constants and logarithmic factors for
  reconstruction from Hankel matrices of samples from
  $\Gamma$. Indeed, we saw in Theorem~\ref{T:ReconTheo} that even the
  coefficients can only be reconstructed from a matrix $F_{A,B}$
  provided that $A,B$ are interpolation sets for $X_\Omega$. Without a
  priori information on $\Omega$, these sets have to be universal and
  $\Upsilon_N$ is the minimal universal set, at least when degree
  reduction and 
  monomiality are requested. Moreover, Example~\ref{Ex:Hyperbola}
  showed that a complexity of $N^2$ is unavoidable for Hankel
  matrices already for $s=2$.
\end{remark}

\begin{remark}
  In the generic case that happens with
  probability one, the complexity is even lower, namely $\sim 2^s N$,
  as pointed out in Example~\ref{Ex:Generic}. Consequently, with
  probability one the number of variables even enters only as a
  constant.
\end{remark}

\begin{thm}\label{T:PronyIdeal}
  Suppose that $N = \# \Omega$ and $A \subset \Gamma$. Then
  a vector $p \in \CC^A$ belongs to the kernel of $F_{\Upsilon_N,A}$ if and
  only if the associated polynomial fulfills $p \in I_\Omega \cap \Pi_A$.
\end{thm}

\begin{proof}
  The standard ``Prony trick'' yields that for any $\beta \in \Upsilon_N$
  we have
  \begin{equation}
    \label{eq:PronyTrick}
    \left( F_{\Upsilon_N,A} \, p \right)_\beta = \sum_{\alpha \in A}
    f(\alpha+\beta) p_\alpha
    = \sum_{\alpha \in A} \sum_{\omega \in \Omega} f_\omega
    e^{\omega^T (\alpha + \beta)} \,  p_\alpha
    = \sum_{\omega \in \Omega} f_\omega \, e^{\omega^T \beta} \, p (x_\omega)
  \end{equation}
  which we can rewrite in vector form as
  \[
  F_{\Upsilon_N,A} \, p = V( X_\Omega,\Upsilon_N )^T F_\Omega \left[
    p(x_\omega) : \omega \in \Omega \right],
  \]
  and since $\rank V( X_\Omega,\Upsilon_N )^T = \rank F_\Omega = \#
  \Omega$, this vector is zero if and only if $p$ vanishes on
  $X_\Omega$.
\end{proof}

\noindent
A slightly closer inspection of the proof shows that we can reformulate
Theorem~\ref{T:PronyIdeal} in even stronger form.

\begin{corollary}\label{C:PronyIdeal}
  The equivalence
  \[
  p \in \ker F_{A,B} \qquad \Leftrightarrow \qquad
  p \in I_\Omega \cap \Pi_B
  \]
  holds if and only if $\Pi_A$ is an interpolation space for $X_\Omega$.
\end{corollary}

\noindent
By Theorem~\ref{T:PronyIdeal}, the algorithm from \citep{Sauer15S}
could immediately be restated for the matrices
\[
F_{n,k} := \left[ f(\alpha+\beta) :
  \begin{array}{c}
    \alpha \in \Upsilon_N \\ |\beta| \le k
  \end{array}
\right],
\]
but since this approach is based on orthogonal decompositions it
requires square roots which makes it inappropriate for a symbolic
environment. Therefore, the next section provides an algorithm that
works in a symbolic and more ``monomial'' environment.

\section{Symbolic algorithms}
\label{sec:SymbolAlgo}
Now we are in position to turn the observations obtained so far into
detailed 
symbolic algorithms for the reconstruction of $f$. The first one
will be called Sparse
Monomial Interpolation with Least Elements (SMILE), in contrast to the
Sparse Homogeneous Interpolation Technique introduced
in\citep{Sauer15S}. Both methods have 
in common that for $k=0,1,\dots$ they \emph{successively} compute
$\Pi_A \cap \Pi_k$ and an H--basis of $I_\Omega \cap \Pi_k$ at the
same time by appropriate update rules. This is more efficient than
first determining some basis of the ideal, then a ``good'' basis
(Gr\"obner or H--basis) and afterwards the quotient space.

We return to the function of the form \eqref{eq:PronyFun} and assume
that $N := \# \Omega$ is known. During the process, we will consider
matrices of the form
\begin{equation}
  \label{eq:FkDef}
  F_k := \left[ f(\alpha+\beta) :
    \begin{array}{c}
      \alpha \in \Upsilon_N \\ \beta \in A_k
    \end{array}
  \right], \qquad A_k \subseteq \Gamma_k,  
\end{equation}
with nested sets $A_0 \subseteq A_1 \subseteq \cdots$ to be determined
during the reconstruction process which will eventually terminate for
some $n$ with
$A_n$ being a monomial degree reducing set for interpolation at
$X_\Omega$. The goal is to decompose $\Gamma_k$ into three sets $A_k$,
$B_k$ and $I_k$ where $A_k$ contains the exponents from $A \cap
\Gamma_k$, where $A$ is the interpolation space to be constructed
eventually. $I_k$ contains the leading powers of an H-basis of the
ideal and $B_k$ multiindices from $I_k + ( \NN_0^s \setminus \{ 0 \}
)$. Since the kernel of $F_k$ consists of the coefficient vectors from
an ideal, polynomials with a leading term from $B_k$ can be ignored in
the inductive step which improves the performance of the algorithm.

The initialization is $A_0 = \{ 0 \}$ which leaves us with
\[
F_0 = \left[ f(\alpha) : \alpha \in \Upsilon_N \right] \in \CC^{\#
  \Upsilon_N \times 1}.
\]
which is $\neq 0$ since
$F_0 = V (X_\Omega, \Upsilon_N )^T F_\Omega 1_N$ with $\rank V (
X_\Omega, \Upsilon_N ) = \# \Omega = N$ and $F_\Omega \neq 0$. Hence,
$\rank F_0 = 1 = \# A_0$. Moreover, we define $I_0 = B_0 := \emptyset$ as a
subset of $\Gamma_0$ and note that $\Gamma_0 = A_0 \cup I_0 \cup B_0$.

To advance from $k \to k+1$ we assume that $\rank F_k = \# A_k$ and
$\Gamma_k = A_k \cup B_k \cup I_k$ and
define the sets
\[
B := \bigcup_{j=1}^s \left( A_k^c \cap \Gamma_k^0 \right) + \epsilon_j
\subseteq \Gamma_{k+1}^0, \qquad
\widetilde A_{k+1} := A_k \cup ( \Gamma_{k+1}^0 \setminus
    B )
\]
and extend $F_k$ into
\[
\widetilde F_{k+1} := \left[ f(\alpha+\beta) :
  \begin{array}{c}
    \alpha \in \Upsilon_N \\ \beta \in \widetilde A_{k+1}
  \end{array}
\right] = \left[ F_k \,|\, G \right], \qquad G \in \CC^{\Upsilon_N \times
  (\widetilde A_{k+1} \setminus A_k)},
\]
with additional columns. Next, we compute a basis of
\[
\ker \widetilde F_{k+1} = \left\{ y \in \CC^{\widetilde A_{k+1}}
  \setminus \{ 0 \} : \widetilde F_{k+1} y = 0 \right\}
\]
and arrange it into a matrix $Y \in \CC^{\widetilde
  A_{k+1} \times d}$ with $d := \dim \ker \widetilde F_{k+1} \le \#
(\widetilde A_{k+1} \setminus A_k)$. We
write
\[
Y = \left[
  \begin{array}{c}
    Y_k \\ Y'
  \end{array}
\right], \qquad Y_k \in \CC^{A_k \times d}, \quad Y' \in
\CC^{\widetilde A_{k+1} \setminus A_k \times d},
\]
and obtain that
\begin{equation}
  \label{eq:Fk+1Y}
  0 = \widetilde F_{k+1} Y = \left[ F_k \,|\, G \right] \, \left[
    \begin{array}{c}
      Y_k \\ Y'
    \end{array}
  \right] = F_k Y_k + G Y'.
\end{equation}
Since $F_k$ is of maximal rank by assumption, it has a left
inverse, for example the pseudoinverse $F_k^+$, which can be computed
symbolically, cf. \citep{springer87:_gener}, giving $Y_k = -F_k^+
G Y'$ by \eqref{eq:Fk+1Y}. This implies the Schur complement relation
\begin{equation}
  \label{eq:KernelFindSimpler}
  0 = \widetilde F_{k+1} \left[
    \begin{array}{c}
      -F_k^+ G \\ I
    \end{array}
  \right] Y' = \left[ F_k \,|\, G \right] \left[
    \begin{array}{c}
      -F_k^+ G \\ I
    \end{array}
  \right] Y' = 
  ( I - F_k F_k^+ ) G Y'.  
\end{equation}
Still, $\rank Y' = d$, hence there exist $d$ linear independent rows
of $Y'$ or a permutation $P$ such that
\[
Z := \left[ I_{d \times d} \,|\, 0 \right] P Y' \in \RR^{d \times d}
\]
is invertible, so that
\[
P Y' Z^{-1} = \left[
  \begin{array}{c}
    I \\ *
  \end{array}
\right].
\]
After replacing $Y'$ by $Y' Z^{-1}$ and ordering the elements of
$\widetilde A_{k+1} \setminus A_k$ according to the permutation $P$,
we can thus assume that 
$Y' = \left[
  \begin{array}{c}
    I_{d \times d} \\ *
  \end{array}
\right]$ and determine $Y_k = -F_k^* G Y'$ which of course also
requires a compatible ordering of the columns of $G$. Now we set
\begin{eqnarray}
  \label{eq:Ak+1Def}
  A_{k+1} & := & A_k \cup \left( \widetilde A_{k+1} \setminus A_k \right)
  \left( d+1:\# (\widetilde A_{k+1} \setminus A_k) \right), \\
  \label{eq:Ik+1Def}
  I_{k+1} & := & I_k \cup \left( \widetilde A_{k+1} \setminus A_k
  \right) (1:d), \\
  \label{eq:Bk+1Def}
  B_{k+1} & := & B_k \cup B,
\end{eqnarray}
where the components of the vectors are indexed in a Matlab--like
way. It follows directly from \eqref{eq:Ak+1Def}, \eqref{eq:Ik+1Def},
\eqref{eq:Bk+1Def}
and the assumption on $A_k$ and $I_k$ that $A_{k+1} \cup I_{k+1} =
\Gamma_{k+1}$. For $\alpha \in \left( \widetilde A_{k+1} \setminus A_k
\right) (1:d)$ we define polynomials $q_\alpha \in \Pi_{\widetilde A_k}$ whose
coefficients are the respective columns of $Y$.

Having determined $A_{k+1}$, we can build the matrix $F_{k+1}$
according to \eqref{eq:FkDef}. If $A_{k+1}$ is a proper superset of
$A_k$, the matrix enlarges $F_k$ by adding further columns which
immediately yields that $\rank A_{k+1} \ge \rank A_k$. 
This construction also advances the rank hypothesis from $k$ to $k+1$.

\begin{lemma}\label{L:Fk+1Rank}
  The matrix
  \begin{equation}
    \label{eq:Fk+1Rank}
    F_{k+1} = F_{\Upsilon_N,A_{k+1}} = V (\Upsilon_N,X_\Omega)^T F_\Omega V (
    A_{k+1},X_\Omega )
  \end{equation}
  has maximal rank $\# A_{k+1}$.
\end{lemma}

\begin{proof}
  The claim is the induction hypothesis if $A_{k+1} = A_k$ and
  therefore trivial in this case. If $\# A_{k+1} > \# A_k$ we first
  note that, since $\rank F_k = \# A_k$ there exists a matrix
  \[
  Z = \left[
    \begin{array}{c}
      * \\ 0_{\# ( \widetilde A_{k+1} \setminus A_k ) \times \# A_k}
    \end{array}
  \right] \in \CC^{\# \widetilde A_{k+1} \times \# A_k}
  \]
  such that $\rank \widetilde F_{k+1} Z = \# A_k$. If we extend the
  matrix $Y$ from above as
  \[
  Z' = [ Y \,|\, \hat Y ] = \left[
    \begin{array}{cc}
      * & * \\
      I_{d\times d} & 0 \\
      * & *
    \end{array}
  \right] \in \RR^{\# ( \widetilde A_{k+1} \setminus A_k ) \times (
    \widetilde A_{k+1} \setminus A_k )}
  \]
  into a matrix of rank $\widetilde A_{k+1} \setminus A_k$, the fact
  that $Y$ exactly contains the kernel of $\widetilde F_{k+1}$ implies
  that
  \[
  \rank F_{k+1} [ Z \,|\, \widehat Y ] =
  \rank \widetilde F_{k+1} [ Z \,|\, Y \,|\, \widehat Y ] =
  \rank \widetilde F_{k+1} [ Z \,|\,\widehat Y ] = \# \widetilde A_{k+1} - d
  = \# A_{k+1}
  \]
  and therefore $\rank F_{k+1} = \# A_{k+1}$.
\end{proof}

This algorithm is repeated until $A_{k+1} = A_k$ and it solves Prony's
problem at termination. Let us summarize the algorithm formally.

\begin{algorithm}[Prony's method, symbolic
  decomposition]\label{Alg:PronySymbolic}~
  \par\noindent
  \textbf{Given:} function $f : \Gamma \to \CC$ and $N \ge 0$.
  \begin{enumerate}
  \item Initialization: $A_0 := \{ 0 \}$, $I_0 := \emptyset$, $B_0 :=
    \emptyset$ and $F_0 := \left[ f(\alpha) : \alpha \in \Upsilon_N
    \right]$.
  \item For $k=0,1,\dots$ repeat
    \begin{enumerate}
    \item Compute
      \[
      B := \bigcup_{j=1}^s \left( A_k^c \cap \Gamma_k^0 \right) + \epsilon_j
      \]
      set $b = \# ( \Gamma_{k+1}^0 \setminus B )$ and compute
      \[
      G := \left[ f (\alpha+\beta) :
        \begin{array}{c}
          \alpha \in \Upsilon_N \\ \beta \in \Gamma_{k+1}^0 \setminus B
        \end{array}
      \right].
      \]
    \item Determine the kernel of $( I - F_k F_k^+ ) G$ and write it
      as a matrix $Y = \left[
        \begin{array}{c}
          I_{d \times d} \\ *
        \end{array}
      \right]$, where $d$ is the dimension of the kernel. This defines
      an ordering of $\Gamma_{k+1}^0 \setminus B$.
    \item Set
      \begin{eqnarray*}
        A_{k+1} & = & A_k \cup ( \Gamma_{k+1}^0 \setminus B ) (d+1:b ),
        \\
        I_{k+1} & = & I_k \cup ( \Gamma_{k+1}^0 \setminus B ) ( 1:d ),
        \\
        B_{k+1} & = & B_k \cup B.
      \end{eqnarray*}
    \item Set
      \[
      F_{k+1} := \left[
        f(\alpha+\beta) :
        \begin{array}{c}
          \alpha \in \Upsilon_N \\ \beta \in A_{k+1}
        \end{array}
      \right] = \left[ F_k \, \left| f(\alpha+\beta) :
          \begin{array}{c}
            \alpha \in \Upsilon_N \\ \beta \in A_{k+1} \setminus A_k
          \end{array}
        \right.
      \right].
      \]
    \item Define polynomials $q_\alpha$, $\alpha \in I_{k+1} \setminus
      I_k$, by taking the $\alpha$th column of the matrix
      $\left[
        \begin{array}{c}
          - F_k^+ G Y \\ Y
        \end{array}
      \right]$ as coefficient vectors.
    \end{enumerate}
    until $A_{k+1} = A_k$.
  \end{enumerate}
  \textbf{Results:} Monomial degree reducing interpolation space
  $\Pi_{A_k}$ for $X_\Omega$ and H--basis $H = \left\{ h_\alpha :
    \alpha \in I_{k+1} \right\}$ for $I_\Omega$. 
\end{algorithm}

\begin{remark}
  The intuitive meaning of \eqref{eq:Ak+1Def}, \eqref{eq:Ik+1Def} and
  \eqref{eq:Bk+1Def} is to split the multiindices from
  $\Gamma_{k+1}^0$ into three groups: $A_{k+1}$ collects those which
  are used for the interpolation space, $I_{k+1}$ those which are
  needed for the H--basis and $B_{k+1}$ those which also belong to the
  ideal due to an H--basis element of lower degree.
\end{remark}

\begin{thm}\label{T:PronyWorks}
  If $N \ge \# \Omega$ Algorithm~\ref{Alg:PronySymbolic}
  \begin{enumerate}
  \item terminates at some level $k = n \le \#\Omega$,
  \item determines a degree reducing interpolation set $A_n$ for $X_\Omega$,
  \item determines an H--basis $H := \{ h_\alpha : \alpha \in I_{n+1} \}$
    for $I_\Omega$.
  \end{enumerate}
\end{thm}

\begin{proof}
  We will first verify that the polynomial set
  \begin{equation}
    \label{eq:HkDef}
    H_k := \left\{ (\cdot)^\beta h_\alpha : \beta \in \Pi_{k - |\alpha|},
      \alpha \in I_k \right\}, \qquad k \in \NN_0,
  \end{equation}
  forms a vector space basis for $I_\Omega \cap \Pi_k$. This is
  trivially true as long as $I_k = \emptyset$. Since, for any $\alpha
  \in I_k$, we have $\deg h_\alpha =
  |\alpha|$ and since the coefficient vector
  of any $h_\alpha$ belongs to $\ker F_{|\alpha|-1}$,
  Theorem~\ref{T:PronyIdeal} ensures that $h_\alpha \in I_\Omega \cap
  \Pi_{|\alpha|}$. Thus, $H_k \subseteq I_\Omega
  \cap \Pi_k$. For the converse inclusion inclusion, we assume that $q
  \in \Pi_k \cap I_\Omega$ with $\deg q = k$. Again by
  Theorem~\ref{T:PronyIdeal}, the coefficient vector of $q$ has to
  belong to $\ker F_k$. By subtracting a proper element of $H_k$ we
  can, like in the standard Gr\"obner basis division algorithm,
  eliminate all monomials from $\Gamma_k \setminus A_k$ from $q$ and
  obtain another $q' \in I_\Omega \cap \Pi_k$ since we only subtracted
  ideal elements. Moreover, $q' \in \ker F_k$, but since, according to
  Lemma~\ref{L:Fk+1Rank}, $F_k$ has rank $\# A_k$, the only
  polynomial $q' \in \Pi_{A_k} \cap I_\Omega$ is $q' = 0$. Hence, $q
  \in H_k$ and therefore $\Span H_k \supseteq I_\Omega
  \cap \Pi_k$ as well, yielding $\Span H_k = I_\Omega \cap
  \Pi_k$.

  Since $\Pi$ is a Noetherian ring, cf. \citep{CoxLittleOShea92}, there
  exists some $n \in \NN_0$
  such that the increasing chain $\left\langle H_k \right\rangle$, $k
  \in \NN_0$, of ideals from \eqref{eq:HkDef} stabilizes, and
  since the strict inclusion $I_{k+1} \supset I_k$ implies $H_{k+1}
  \supset H_k$ in the strict sense, it follows that also $I_n =
  I_{n+1} = \cdots$. By \eqref{eq:Ik+1Def} this means that either
  $d=0$, i.e., all columns of $Y$ correspond to ideal elements, or $B
  = \Gamma_{n+1}^0$. In both cases we have that $\Lambda (H_{n+1})
  \cap \Pi_{n+1}^0 = \Pi_{n+1}^0$ and $A_n = A_{n+1} =
  \cdots$. Since $\Pi_n = A_n \oplus H_n$, it follows that $\Pi_{A_n}$
  is a degree reducing interpolation space, see \citep{Sauer06a} and that
  \[
  H := \{ h_\alpha : \alpha \in I_n \}
  \]
  is an H--basis for $I_\Omega$.
\end{proof}

\noindent
The algorithm can also be formulated in a ``term-by-term'' way which
gives an implicit version of the M\"oller--Buchberger algorithm from
\citep{BuchbergerMoeller82}. To that end, we recall the classical
\emph{graded lexicographic ordering} ``$\preceq$'' where $\alpha \prec \beta$
provided that $|\alpha| < |\beta|$ or $|\alpha| = |\beta|$ and there
exists some $k \in \{1,\dots,s\}$ such that
$\alpha_j = \beta_j$, $j=1,\dots,k-1$, and $\alpha_k < \beta_k$,
cf. \citep{CoxLittleOShea92}. ``$\preceq$'' is a total ordering on
$\Gamma$ and therefore also induces a total order on the monomials or
\emph{terms} $(\cdot)^\alpha$, $\alpha \in \Gamma$. The algorithm,
based on Sparse Monomial Interpolation with Least Elements (SMILE),
now proceeds as follows. 

\begin{algorithm}[Prony's method, SMILE]\label{Alg:PronySymbolicMono}~
  \par\noindent
  \textbf{Given:} function $f : \Gamma \to \CC$ and $N \ge \# \Omega$.
  \begin{enumerate}
  \item Initialization: $B = \Gamma_{N+1}$, $I := \emptyset$, $A =
    \emptyset$, $F =[] \in
    \CC^{\# \Upsilon_N \times 0}$.
  \item While $B \neq \emptyset$
    \begin{enumerate}
    \item $\beta := \min_{\preceq} B$.
    \item Expand the matrix by one column:
      \[
      \widetilde F = \left[ F \,|\, f(\alpha + \beta) :
        \alpha \in \Upsilon_N \right].
      \]
    \item If $\rank \widetilde F = \rank F$ then determine $0 \neq
      q_\beta \in \ker \widetilde F$ and set
      \begin{eqnarray*}
        B & := & B \setminus \left( \beta + \Gamma_{N+1-|\beta|} \right),
        \\
        I & := & I \cup \{ \beta \}.
      \end{eqnarray*}
      If $\rank \widetilde F > \rank F$ then set $F := \widetilde
      F$ and
      \begin{eqnarray*}
        B & := & B \setminus \{ \beta \}, \\
        A & := & A \cup \{ \beta \}.
      \end{eqnarray*}
    \end{enumerate}
  \end{enumerate}
  \textbf{Results:} Gr\"obner basis $\{ q_\alpha : \alpha \in I \}$
  and monomial quotient space $\Pi_A \simeq \Pi / I_\Omega$.
\end{algorithm}

\noindent
The proof of the validity of this algorithm works like the proof of
the preceding theorem, one only has to keep in mind that whenever the
rank increases, a new term for the quotient space has been found which
guarantees that always $\rank F = \# A$. If, on the other hand, the
rank does not increase after adding the column, there must be a
nontrivial kernel element, unique up to normalization with nonzero
value in its $\beta$th component which becomes a member of the ideal
basis.

\begin{thm}\label{T:GBAlgoEvals}
  Algorithm~\ref{Alg:PronySymbolicMono} computes the decomposition
  using at most
  \[
  s \, N^2 \log^{s-1} N
  \]
  evaluations of $f$ if $N = \# \Omega$.
\end{thm}

\begin{proof}
  Each column added needs at most $\# \Upsilon_N \le N \log^{s-1} N$
  evaluations of $f$. The number of columns
  added during the algorithm is
  \[
  \# A + \# I \le \# A + \# \partial A \le \# A + (s-1) \# A = s \, \#
  A
  \]
  since $I \subseteq \partial A$.
\end{proof}
\noindent
Once the set $A_n$ and the H--basis $H$ are determined, the
points $X_\Omega$ can be determined by means of \emph{multiplication
  tables} as described in \citep{AuzingerStetter88,MoellerStetter95}
and efficiently determined by the methods from
\citep{moeller01:_multiv}. For reduction we can again use the inner
product from Lemma~\ref{L:MonHBasReduced}. Once $X_\Omega$ and thus
$\Omega$ are determined, the coefficients $f_\omega$, $\omega \in
\Omega$, are determined by solving a linear system, for details see
\citep{Sauer15S}. 

\section{Sparse polynomials}
\label{sec:SparsePoly}
A problem, closely related to Prony's problem is the reconstruction of
\emph{sparse polynomials}, i.e., of polynomials of the form
\[
f(x) = \sum_{\kappa \in K} f_\kappa \, x^\kappa, \qquad f_\kappa \in
\CC \setminus \{ 0 \}, \quad \kappa \in K,
\]
where \emph{sparsity} means that $\# K$ is (very) small relative to
${\deg K + s \choose s} = \Pi_{\deg K}$. This requirement is quite
easy to achieve in several variables.

The ``classical'' method to reconstruct $f$ from samples on $\Gamma$
is the one from \citep{ben-or88} and uses a univariate Prony method
together with divisibility aspects of relatively prime numbers. A
variant with unit roots and the Chinese remainder theorem can be found
in \citep{giesbrecht09:_symbol}.

As shown in \citep{Sauer15S}, it is easy to reduce this problem to
Prony's problem: let $\Theta \in \ZZ^{s \times s}$ be any nonsingular
matrix, then
\[
f \left( 2^{\Theta \alpha} \right) = \sum_{\kappa \in K} f_\kappa \,
e^{\log 2 \, (\Theta \kappa)^T \alpha} = \sum_{\kappa \in K} f_\kappa \,
e^{\omega_\kappa^T \alpha}, \qquad \omega_\kappa := \log 2 \, \Theta \kappa,
\]
which is Prony's problem with $\Omega = \{ \omega_\kappa : \kappa \in
A \}$ which can be solved by considering the Hankel matrices
\[
F_{A,B} := 
\left[ f \left( 2^{\Theta ( \alpha + \beta)} \right):
  \begin{array}{c}
    \alpha \in A \\ \beta \in B
  \end{array}
\right].
\]
If the coefficients $f_\kappa$ of $f$ belong to the \emph{Gaussian
  integers} $\ZZ + i \ZZ$, which is the normal assumption in symbolic
computations, the evaluations in $F_{A,B}$ are rational numbers and
therefore also the ideal basis computed in the preceding section
consists of \emph{symbolic polynomials} with coefficients in $\QQ +
i\QQ$. The same holds true for the multiplication tables and only the
joint eigenvalues have to be computed in numerical precision giving
the frequencies $\omega_\kappa$ from which the exponents can be
computed as
\[
\kappa = \mathop{\rm rd} \left( \frac{1}{\log 2} \Theta^{-1}
  \omega_\kappa \right)
\]
by rounding to the next integer.

\section{Appendix: Two facts on interpolation spaces}
\label{sec:Appendix}
This section gives a detailed exposition of some of the algebraic
results used in the preceding ones. We begin by pointing out that any monomial
degree reducing interpolation space automatically defines a natural H--basis.

\begin{lemma}\label{L:InterpolDuality}
  If $A \subset \Gamma$ is a finite set such that $\Pi_A$ is a degree
  reducing interpolation space for $X$ then the 
  polynomials $q_\alpha := (\cdot)^\alpha - L_A (\cdot)^\alpha$, $\alpha
  \in \partial A$, form an H--basis of $I_X$.
\end{lemma}

\begin{proof}
  Define 
  \begin{equation}
    \label{eq:LInterpolDualityPf1}
    q_\alpha := (\cdot)^\alpha - L_A (\cdot)^\alpha, \qquad \alpha \in \Gamma,
  \end{equation}
  and note that $q_\alpha = 0$ for $\alpha \in A$ and $\deg
  q_\alpha = |\alpha|$ since $A$ is degree reducing. Moreover,
  $\frac{\partial^\beta q_\alpha}{\partial x^\beta} (0) =
  \alpha! \delta_{\alpha,\beta}$, $\alpha,\beta \in A^c := \Gamma \setminus A$.
  Therefore the polynomials
  $q_\alpha$, $\alpha \in \Gamma_n \setminus A$, and $(\cdot)^\alpha$,
  $\alpha \in A_n := A \cap \Gamma_n$
  form a basis of $\Pi_n$ for any $n \in \NN$. Consequently, any
  polynomial $p = \sum p_\alpha \, (\cdot)^\alpha \in \Pi$ can be
  written as
  \begin{eqnarray*}
    p(x) & = & L_A p (x) + p(x) - L_A p(x) = L_A p (x) + \sum_{\alpha
      \in A^c} p_\alpha \, q_\alpha (x)
  \end{eqnarray*}
  and
  \begin{equation}
    \label{eq:LInterpolDualityPf2}
    p \in I_X \qquad \Leftrightarrow \qquad p(x) = \sum_{\alpha
    \in A^c} p_\alpha \, q_\alpha (x).
  \end{equation}
  The representation on the right hand side of
  \eqref{eq:LInterpolDualityPf2} is an H--representation
  \citep{MoellerSauer00}, hence the polynomials
  $\{ q_\alpha : \alpha \in A^c \}$, form an infinite H--basis, and
  $\left\{ q_\alpha : \alpha \in A_{n+1}^c \right\}$, $A_{n+1}^c := A^c \cap
  \Gamma_{n+1}$, where $n := \deg A = \max \{ |\alpha| : \alpha \in A
  \}$, is a finite H--basis of $I_X$.

  Next, we fix $j \in \{ 1,\dots,s \}$ and $\alpha \in A^c$, and write
  $p (x) := L_A (\cdot)^\alpha (x) \in \Pi_A$ as $p(x) = \sum p_\beta
  \, x^\beta$. Then,
  \begin{eqnarray*}
    \lefteqn{q_{\alpha+\epsilon_j} (x) - x_j q_\alpha (x) =
      x^{\alpha+\epsilon_j} - L_A (\cdot)^{\alpha+\epsilon_j} (x) - 
      x^{\alpha+\epsilon_j} + x_j \, L_A (\cdot)^\alpha (x)} \\
    & = & x_j \, L_A (\cdot)^\alpha (x) - L_A
    (\cdot)^{\alpha+\epsilon_j} (x)
    = \sum_{\beta \in A} p_\beta \, x^{\beta + \epsilon_j} - L_A
    (\cdot)^{\alpha+\epsilon_j} (x) \\
    & = & \sum_{\beta \in \partial A} p_{\beta-\epsilon_j} \, x^\beta
    + \sum_{\beta \in A \cap (A+\epsilon_j)} p_{\beta-\epsilon_j} \,
    x^\beta - L_A (\cdot)^{\alpha+\epsilon_j} (x) \\
    & = & \sum_{\beta \in \partial A} p_{\beta-\epsilon_j} \, q_\beta (x) + \sum_{\beta
      \in \partial A } p_{\beta-\epsilon_j} \, L_A (\cdot)^\beta (x) +
    \sum_{\beta \in A \cap (A+\epsilon_j)} p_{\beta-\epsilon_j} \,
    x^\beta - L_A (\cdot)^{\alpha+\epsilon_j} (x) \\ 
    & = & \sum_{\beta \in \partial A} p_{\beta-\epsilon_j} \, q_\beta
    (x) + \tilde p (x)
  \end{eqnarray*}
  with some $\tilde p \in \Pi_A$. The polynomial on the left hand side
  belongs to the ideal and vanishes on $X$ as do the $q_\beta$ in the
  sum on the right hand side, hence $\tilde p(X) = 0$ and therefore,
  taking account on the lengths of the $\beta$ appearing in the above
  decomposition, 
  \begin{equation}
    \label{eq:LInterpolDualityPf3}
    q_{\alpha+\epsilon_j} (x) - x_j q_\alpha (x) = \sum_{\beta
      \in \partial A \cap \Gamma_{n+1}} c_{\alpha,\beta} \, q_\beta (x), \qquad
    c_{\alpha,\beta} \in \CC.
  \end{equation}
  Therefore, any $q_\alpha$ with $\alpha \in A^c$ such that $\alpha -
  \epsilon_j \in A^c$ has a H--representation by $(\cdot)_j
  q_{\alpha-\epsilon_j}$ and $q_\beta$, $\alpha \in \partial A \cap
  \Gamma_n$. If, on the other hand, $\alpha - \epsilon_j \not\in A^c$,
  $j=1,\dots,s$, 
  and $\alpha \neq 0$, then there must be some some $j$ such that
  $\alpha - \epsilon_j \in A$ and therefore $\alpha \in \partial
  A$. Since any nontrivial degree reducing set must contain $0$, it
  follows that $0 \not\in A^c$ and therefore an inductive application
  of the above process shows that any $q_\alpha$ must eventually be
  written as a linear combination of $q_\beta$, $\beta \in \partial A
  \cap \Gamma_{|\alpha|}$.
\end{proof}

\noindent
We recall the notion of a reduced polynomial. Given an inner product
$(\cdot,\cdot)$ on
$\Pi$, we call a polynomial $p$ \emph{reduced} if each homogeneous term
\[
p_j (x) := \sum_{|\alpha| = j} p_\alpha x^\alpha, \qquad
j=0,\dots,\deg p,
\]
of $p$ is perpendicular to the homogeneous leading forms in $\Lambda
(I_X) \cap \Pi_j^0$, where $\Lambda (p) := p_{\deg p} \in \Pi_{\deg p}^0$.
As shown in
\citep{Sauer01} that whenever $H$ is an H--basis for $I_X$ there
exists, for any polynomial $p \in \Pi$, a decomposition
\begin{equation}
  \label{eq:HDecomp}
  p = \sum_{h \in H} q_h \, h + r, \qquad \deg q_h + \deg h \le \deg p,  
\end{equation}
such that $r$ is reduced and depends only on $I_X$ and
$(\cdot,\cdot)$ and is zero if and only if $p \in I_X$. Therefore, $r$
can be seen as a well defined mapping $r : \Pi \to \Pi$. Also note that
\eqref{eq:HDecomp} is the multivariate analog of \emph{euclidean
  division} or \emph{division with remainder} and that $r$ is the
natural interpolant of $p$.

\begin{lemma}\label{L:MonHBasReduced}
  If $A \subset \Gamma$ is a finite set such that $\Pi_A$ is a degree
  reducing interpolation space for $X$ then there exists an inner
  product $(\cdot,\cdot)$ such that $\Pi_A = r(\Pi)$.
\end{lemma}

\begin{proof}
  We use the H--basis $q_\alpha$, $\alpha \in A^c$, defined in
  \eqref{eq:LInterpolDualityPf1} and define the inner product
  separately on $\Pi_n^0 \times \Pi_n^0$, $n \in \{ 0,1,\dots \}$. If
  $n < \min
  \{ |\alpha| : \alpha \in A^c \}$ and $n > \deg A$, we simply use the
  inner product of the coefficients,
  \[
  (p,p')_n := \sum_{|\alpha| = n} \overline{p_\alpha} p_\alpha', \qquad p,p' \in
  \Pi_n^0.
  \]
  For other values of $n$ we first observe that the polynomials
  $x^\alpha$, $\alpha \in A \cap \Gamma_n^0$ and the leading forms
  $\Lambda (q_\alpha)$, $\alpha \in A^c \cap \Gamma_n^0$, span
  $\Pi_n^0$. We arrange the coefficient vectors into a nonsingular matrix
  $Y \in \CC^{r_n^0 \times r_n^0}$ where $r_n^0 := \dim \Pi_n^0 =
  {n+s-1 \choose s-1}$ and note that the Gramian $G := Y Y^H$ is hermitian
  and positive definite. Defining
  \[
  (p,p')_n = p^H G^{-1} p = \sum_{|\alpha| = |\beta| = n} ( G^{-1}
  )_{\alpha,\beta} p_\alpha p_\beta', \qquad p,p' \in
  \Pi_n^0,
  \]
  we get that
  \[
  (Y,Y)_n = Y^H G^{-1} Y = Y^H ( Y Y^H )^{-1} Y = I
  \]
  which means that the vectors $e_\alpha$, $\alpha \in A \cap
  \Gamma_n^0$ are perpendicular to the coefficient vectors of $\Lambda
  (q_\alpha)$, $\alpha \in A^c \cap \Gamma_n^0$. Consequently, the
  inner product
  \[
  (p,p') = \sum_{j \in \NN_0} ( p_j,p_j' )_j, \qquad p,p' \in \Pi,
  \]
  has the property that a polynomial is reduced if and only if it
  belongs to $\Pi_A$, that is, $\Pi_A = r(\Pi)$ as claimed.
\end{proof}

\section*{Acknowledgement}
I want to thank the referees and the editor for their careful reading and
helpful suggestions and references.

\section*{References}


\end{document}